\documentclass[10pt,twoside]{amsart}
\usepackage{amsmath, amsthm, amscd, amsfonts, amssymb, graphicx}%, color}
\newtheorem{thm}{Theorem}[section]

\newtheorem{defn}[thm]{Definition}
\newtheorem{rem}[thm]{\bf{Remark}}

\newtheorem{problem}[thm]{Example}
\numberwithin{equation}{section}

\begin{document}
%\leftline{ \scriptsize \it Bulletin of the Iranian MathematicalSociety  Vol. {\bf\rm XX} No. X {\rm(}201X{\rm)}, pp XX-XX.}

%\vspace{1.0 cm}

\begin{center}\large{{\bf{Composite mapping on hesitant fuzzy soft classes}}}

\vspace{0.5cm}

%{\bf{
Manash Jyoti Borah$^{1}$ and Bipan Hazarika$^{2\ast}$
%}}%$^{1}$ and {\bf{Syed Abdul Mohiuddine}}$^{2}$

\vspace{.2cm}
$^{1}$Department of Mathematics, Bahona College,  Jorhat-785 101, Assam, India\\
Email:mjyotibora9@gmail.com

\vspace{.2cm}

$^{2}$Department of Mathematics, Rajiv Gandhi University, Rono Hills, Doimukh-791 112, Arunachal Pradesh, India\\

Email:  bh\_rgu$@$yahoo.co.in
\end{center}
\vspace{.5cm}
\title{}
\author{}
\thanks{{$^{\ast}$The corresponding author}}

\begin{abstract}
% {\bf Abstract:}
In this paper, we study composite mapping on hesitant fuzzy soft classes. We also establish some interesting properties of this notion and support them with examples. Also definitions of fuzzy soft image and fuzzy soft inverse image put forward by Kharal and Ahmad \cite{kharal} have been reintroduced in hesitant fuzzy soft classes.\\

%{\bf Keywords:}

AMS subject classification no: 03E72.
\end{abstract}

\maketitle
\pagestyle{myheadings}
\markboth{\rightline {\scriptsize Borah, Hazarika}}
         {\leftline{\scriptsize Composite mapping on hesitant fuzzy soft ...}}

\maketitle
\section{Introduction}

The Hesitant fuzzy set, as one of the extension of Zadeh \cite{zadeh} fuzzy set, allows the membership degree that an element to a set
presented by several possible values, and  it can express the hesitant information
more comprehensively than other extensions of fuzzy set.  Torra and Narukawa \cite{torra2} introduced the
concept of hesitant fuzzy set.  Xu and Xia \cite{xu}
defined the concept of hesitant fuzzy element, which can be considered as
the basic unit of a hesitant fuzzy set, and is a simple and effective tool used to express the
decision makers’ hesitant preferences in the process of decision making. So many  researchers  has done lots of research work on aggregation, distance,
similarity and correlation measures, clustering analysis, and decision making with
hesitant fuzzy information.  Babitha and John \cite{babitha} defined another important soft set Hesitant fuzzy soft sets. They introduced basic operations such as intersection, union, compliment  and De Morgan's law was proved.  Wang, et al. \cite{wang} applied hesitant fuzzy soft sets in multicriteria group decision making problems.   \\

There are many theories like theory of probability, theory of fuzzy sets, theory of
intuitionistic fuzzy sets, theory of rough sets etc. which can be considered as mathematical
tools for dealing with uncertain data, obtained in various fields of engineering, physics,
computer science, economics, social science, medical science, and of many other diverse
fields. But all these theories have their own difficulties. The most appropriate theory for
dealing with uncertainties is the theory of fuzzy sets, introduced by L.A. Zadeh \cite{zadeh} in 1965.
This theory brought a paradigmatic change in mathematics. But there exists difficulty, how to
set the membership function in each particular case. The theory of intuitionistic fuzzy sets (see \cite{atanassov}) is a
more generalized concept than the theory of fuzzy sets, but this theory has the same
difficulties. All the above mentioned theories are successful to some extent in dealing with
problems arising due to vagueness present in the real world. But there are also cases where
these theories failed to give satisfactory results, possibly due to inadequacy of the
parameterization tool in them. As a necessary supplement to the existing mathematical tools
for handling uncertainty, in 1999, Molodtsov \cite{molodstov} initiated the theory of soft sets as
a new mathematical tool to deal with uncertainties while
modelling the problems in engineering, physics, computer
science, economics, social sciences, and medical sciences.
In \cite{MolodtsovLeonov}, Molodtsov et al. successfully applied soft sets in
directions such as smoothness of functions, game theory,
operations research, Riemann integration, Perron integration, probability, and theory of measurement. Maji et al. \cite{majietal1} gave the first practical application of soft sets in
decision-making problems.  Maji et al. \cite{majietal2} defined
and studied several basic notions of the soft set theory. Also \c{C}a\v{g}man et al. \cite{cagman} studied several basic notions of the soft set theory. V. Torra \cite{torra,torra2} and Verma and Sharma \cite{verma} discussed the relationship between hesitant fuzzy set and showed that the envelope of hesitant fuzzy set is an intuitionistic fuzzy set. A lot of work has been done about hesitant fuzzy sets, however, little has been done about the hesitant fuzzy soft sets.
In \cite{kharal2}, Kharal and Ahmad introduced the notions of a mapping on the classes of soft sets and studied the properties of soft images. In \cite{kharal}, Kharal and Ahmad introduced the notions of a mapping on the classes of fuzzy soft sets.
\\

In this paper, we study composite mapping on hesitant fuzzy soft classes. We also establish some interesting properties of this notion and support them with examples.

\section{Preliminary Results}

In this section we recall some basic concepts and definitions
regarding fuzzy soft sets, fuzzy soft image,  fuzzy soft inverse image, hesitant fuzzy  set and hesitant fuzzy soft set.

\begin{defn}
\cite{maji} Let $U$ be an initial universe and $F$ be a set of parameters. Let $\tilde{P}(U)$ denote the power set of $U$ and $A$ be a
non-empty subset of $F.$  Then $F_A$ is called a
fuzzy soft set over U where $F:A\rightarrow \tilde{P}(U)$is a
mapping from $A$ into $\tilde{P}(U).$
\end{defn}

\begin{defn}
\cite{molodstov}  $F_E$ is called a soft
set over $U$ if and only if $F$ is a mapping of $E$ into the set of all
subsets of the set $U.$
\end{defn}
In other words, the soft set is a parameterized family of subsets
of the set $U.$ Every set $F(\epsilon),$ $\epsilon\tilde{\in} E,$
from this family may be considered as the set of
$\epsilon$-element of the soft set $F_E$  or as the set of
$\epsilon$-approximate elements of the soft set.
\begin{defn}
	\cite{kharal} Let $X$ be a universe and $E$ a set of attributes. Then the collection of all fuzzy soft sets over $U$ with attributes from $E$ is called a fuzzy soft class and is denoted by $\overline{(X,E)}.$
\end{defn}
\begin{defn}
	\cite{kharal} Let $ \overline{(X,E)} $ and $ \overline{(Y,E')} $ be classes of fuzzy soft sets over $ X $ and $ Y $ with attributes from $ E $ and $ E' $
	respectively. Let $ p:U\longrightarrow V $ and  $ q:E\longrightarrow E' $ be mappings. Then a mappings $ f=(p,q):\overline{(X,E)}\longrightarrow\overline{(Y,E')} $ is defined as follows. \\	For a fuzzy soft set $F_A$ in $\overline{(X,E)},$ $ f(F_A) $ is a fuzzy soft set in $ \overline{(Y,E')} $ obtained as follows: for $ \beta\tilde{\in}q(E)\tilde{\subseteq}E'$ and $ y\tilde{\in}Y, $\\
	\begin {equation*}
	f(F_A)(\beta)(y)=
	\begin {cases}\bigcup_{s\tilde{\in}p^-1(y)}(\bigcup_{\alpha\tilde{\in}q^-1(\beta)\cap A}F_A(\alpha))(x)\\
	& \textrm {if $p^-1(y)\neq \phi,q^-1(\beta)\cap A \neq \phi$ }\\
	0 & \textrm{otherwise}
\end{cases}
\end{equation*}	
$ f(F_A)$  is called a fuzzy soft image of a fuzzy soft set  $F_A.$ 
\end{defn}
\begin{defn}
	\cite{kharal} Let $ p:U\longrightarrow V $ and  $ q:E\longrightarrow E' $ be mappings.Let $ f:\overline{(X,E)}\longrightarrow\overline{(Y,E')} $ be a mapping and $ G_B ,$ a fuzzy soft set in $\overline{(Y,E')}, $ where $ B\tilde{\in}E'. $ Then $ f^{-1}(G_B) ,$ is a fuzzy soft set on $ \overline{(X,E)}, $ defined as follows.  For $ \alpha\tilde{\in}q^{-1}(B)\tilde{\subseteq}E $ and $ x\tilde{\in}U, $\\
	\begin {equation*}
	f^{-1}(G_B)(\alpha)(x)=
	\begin {cases}G_B(q(\alpha))(p(x)),\\
	& \textrm {if $q(\alpha)\tilde{\in}B$ }\\
	0 & \textrm{otherwise}
\end{cases}
\end{equation*}	
$f^{-1}(G_B)  $ is called a fuzzy soft inverse image of $ G_B. $
\end{defn}
\begin{defn}
	\cite{torra} Given a fixed set $X, $ then a hesitant fuzzy set (shortly HFS) in X is in terms of a function that when applied to X return a subset of $[0, 1].$ We express the HFS by a mathematical symbol:\\
	$ F=\{<h, \mu_{F}(x)> : h\in X\}, $ 
	where $ \mu_{F}(x) $ is a set of some values in [0,1], denoting the possible membership degrees of the element $ h\in X $ to the set $ F. $ $ \mu_{F}(x) $ is called a hesitant fuzzy element (HFE) and $ H $ is the set of all HFEs. 
\end{defn}

\begin{defn}
	\cite{torra} Given an hesitant fuzzy set $F, $ define below it lower and upper bound as \\
	lower bound $ F^{-}(x)=\min F(x).$\\
	upper bound $ F^{+}(x)=\max F(x).$
\end{defn}

\begin{defn}
	\cite{torra} Let $ \mu_{1},  \mu_{2}\in H $ and three operations are defined as follows:
	\begin{enumerate}
	\item[(1)] $ \mu_{1}^C=\cup_{\gamma_{1}\in\mu_{1}} \{1-\gamma_{1}\}; $		
	\item[	(2)] $ \mu_{1}\cup\mu_{2}=\cup_{\gamma_{1}\in\mu_{1}, \gamma_{2}\in\mu_{2}}\max \{\gamma_{1},\gamma_{2} \} ; $
				\item[(3)] $ \mu_{1}\cap\mu_{2}=\cap_{\gamma_{1}\in\mu_{1}, \gamma_{2}\in\mu_{2}}\min \{\gamma_{1},\gamma_{2} \}. $		
	\end{enumerate}	
	
\end{defn}
\begin{defn}
	\cite{wang} Let $U$ be an initial universe and $E$ be a set of parameters. Let $\tilde{F}(U)$ be the set of all hesitant fuzzy subsets of $U. $ Then $F_E$ is called a hesitant fuzzy soft set(HFSS) over $U,$  where $\tilde {F}:E\rightarrow \tilde{F}(U). $\\
	A HFSS is a parameterized family of hesitant fuzzy subsets of $U,$ that is, $\tilde{F}(U).$ For all $\epsilon\tilde{\in} E,$  $F(\epsilon)$ is referred to as the set of $ \epsilon -$ approximate elements of the HFSS $F_E.$ It can be written as 
	 $\tilde{F(\epsilon)}=\{<h, \mu_{\tilde{F(\epsilon)(x)}}> : h\in U\}.$\\
	 Since  HFE can represent the situation, in which different membership function are considered possible (see \cite{torra}), $ \mu_{\tilde{F(\epsilon)(x)}} $ is a set of several possible values, which is the hesitant fuzzy membership degree. In particular, if $ \tilde{F(\epsilon)} $ has only one element,  $ \tilde{F(\epsilon)} $  can be called a hesitant fuzzy soft number. For convenience, a hesitant fuzzy soft number (HFSN) is denoted by $ \{<h, \mu_{\tilde{F(\epsilon)(x)}}>\}.$ 
	 
\end{defn}
\begin{problem}\label{exam36}
	Suppose $ U=\{a,b\} $ be an initial universe and $ E=\{e_1, e_2, e_3, e_4\} $ be a set of parameters. Let $ A=\{e_1, e_2\} .$ Then the hesitant fuzzy soft set $ F_A $ is given as  $F_A=\{F(e_1)=\{<a,\{0.6, 0.8\}>, <b, \{0.8, 0.4, 0.9\}>\},  F(e_2)=\{<a,\{0.9, 0.1, 0.5\}>, <b,\{0.2 \}>\}\}.$
\end{problem}
\begin{defn}
	The union of two hesitant fuzzy soft sets $ F_A $ and $ G_B $ over $ (U,E) ,$ is the hesitant fuzzy soft set $ H_C ,$ where $ C=A\cup B $ and  $\forall e\tilde{\in}C,$
		
	$\mu_{H(e)}=\begin{cases}
	\mu_{F(e)}, & \textrm{if $e\tilde{\in}A-B;$}\\
	\mu_{G(e)}, & \textrm{if $ e\tilde{\in}A-B;$}\\
	\mu_{F(e)}\cup \mu_{G(e)}, & \textrm{if $e\tilde{\in}A\cap B.$}\end{cases}$
	
	We write $ F_A \tilde{\cup}G_B=H_C.$
\end{defn}

\section{Mapping on hesitant fuzzy soft classes}

\begin{defn}
Let $U$ be a hesitant fuzzy soft universe and $E$ a set of attributes. Then the collection of all hesitant fuzzy soft sets over $U$ with attributes from $E$ is called a hesitant fuzzy soft class and is denoted by $\overline{(U,E)}.$
\end{defn}
\begin{defn}
Let $ \overline{(U,E)} $ and $ \overline{(V,E')} $ be classes of hesitant fuzzy soft sets over $ U $ and $ V $ with attributes from $ E $ and $ E' $
respectively. Let $ p:U\longrightarrow V $ and  $ q:E\longrightarrow E' $ be mappings. Then a hesitant fuzzy soft mappings $ f=(p,q):\overline{(U,E)}\longrightarrow\overline{(V,E')} $ is defined as follows; \\
For a hesitant fuzzy soft set $F_A$ in $\overline{(U,E)},$ $ f(F_A) $ is a hesitant fuzzy soft set in $ \overline{(V,E')} $ obtained as follows: for $ \beta\tilde{\in}q(E)\tilde{\subseteq}E'$ and $ y\tilde{\in}V, $ $ f(F_A)(\beta)(y)=\bigcup_{\alpha\tilde{\in}q^-1(\beta)\cap A,s\tilde{\in}p^-1(y)} (\alpha)\mu_s$ $ f(F_A)$  is called a hesitant fuzzy soft image of a hesitant fuzzy soft set  $F_A.$ Hence $ (F_A, f(F_A))\tilde{\in}f, $ where $ F_A\tilde{\subseteq}\overline{(U,E)}, f(F_A)\tilde{\subseteq}\overline{(V,E')}. $
\end{defn}
\begin{defn}
If  $ f:\overline{(U,E)}\longrightarrow\overline{(V,E')} $ be a hesitant fuzzy soft mapping, then  hesitant fuzzy soft class $\overline{(U,E)}$ is called the domain of $f $ and the hesitant fuzzy soft class $ \{G_B\tilde{\in}\overline{(V,E')}: G_B=f(H_A), \textrm{for some}  H_A\tilde{\in\overline{(U,E)}}\} $ is called the range of $ f.$The  hesitant fuzzy soft class$\overline{(V,E')} $  
is called co-domain of $ f. $
\end{defn}
\begin{defn}
Let  $ f=(p,q):\overline{(U,E)}\longrightarrow\overline{(V,E')} $ be a hesitant fuzzy soft mapping and $ G_B, $ a hesitant fuzzy soft set in $\overline{(V,E')},$ where $ p:U\longrightarrow V ,  q:E\longrightarrow E' $ and $ B\tilde{\subseteq}E'. $ Then $ f^{-1}(G_B) $ is a hesitant fuzzy soft set in $\overline{(U,E)} $ defined as follows: for $ \alpha\tilde{\in}q^{-1}(B)\tilde{\subseteq}E $ and $ x\tilde{\in}U, $ 
$f^{-1}(G_B)(\alpha)(x)=(q(\alpha))\mu_{p(x)}  $ $f^{-1}(G_B)  $ is called a hesitant fuzzy soft inverse image of $ G_B. $
\end{defn}

\begin{problem}\label{eam35}
Let $ U=\{a,b,c\} $	and $ V=\{x,y,z\}, E=\{e_1, e_2, e_3, e_4\}, E'=\{e_1', e_2', e_3'\} $ and $ \overline{(U,E)},\overline{(V,E')} $ classes of hesitant fuzzy soft sets. Let $ p(a)=y, p(b)=x, p(c)=y $ and $ q(e_1)=e_2', q(e_2)=e_1', q(e_3)=e_2', q(e_4)=e_3'. $ Let $ A=\{e_1, e_2, e_4\} $. Let us  consider a hesitant fuzzy soft set $ F_A $ in $\overline{(U,E)} $ as\\
		$F_A=\{e_1=\{<a,\{0.6, 0.8\}>, <b, \{0.8, 0.4, 0.9\}>,<c,\{0.3\}> \}\\e_2=\{<a,\{0.9, 0.1, 0.2\}>, <b,\{0.5 \}>,<c,\{0.2, 0.4, 0.6\}> \}\\e_4=\{<a,\{0.3\}>, <b,\{0.2,0.6 \}>,<c,\{ 0.4, 0.8\}> \}\}.$\\ 
	Then the hesitant fuzzy soft image of $ F_A $ under  $ f=(p,q):\overline{(U,E)}\longrightarrow\overline{(V,E')} $ is obtained as \\
$f(F_A)(e_1')(x)=\bigcup_{\alpha\tilde{\in}q^-1(e_1')\cap A,s\tilde{\in}p^-1(x)} (\alpha)\mu_s\\
=\bigcup_{\alpha\tilde{\in}\{e_2\},s\tilde{\in}\{b\}} (\alpha)\mu_s\\
=(e_2)\mu_b=\{0.5\}$\\	
$ f(F_A)(e_1')(y)=\bigcup_{\alpha\tilde{\in}q^-1(e_1')\cap A,s\tilde{\in}p^-1(y)} (\alpha)\mu_s\\
=\bigcup_{\alpha\tilde{\in}\{e_2\},s\tilde{\in}\{a,c\}} (\alpha)\mu_s\\
=(e_2)(\mu_a\cup\mu_c)=\{0.2,0.4,0.9\}$\\	
$ f(F_A)(e_1')(z)=\bigcup_{\alpha\tilde{\in}q^-1(e_1')\cap A,s\tilde{\in}p^-1(z)} (\alpha)\mu_s\\
=\bigcup_{\alpha\tilde{\in}\{e_2\},s\tilde{\in}\phi} (\alpha)\mu_s\\
=(e_2)\mu_\phi=\{0.0\}$\\
$ f(F_A)(e_2')(x)=\bigcup_{\alpha\tilde{\in}q^-1(e_2')\cap A,s\tilde{\in}p^-1(x)} (\alpha)\mu_s\\
=\bigcup_{\alpha\tilde{\in}\{e_1, e_3\}\cap \{e_1,e_2,e_4\},s\tilde{\in}\{b\}} (\alpha)\mu_s\\
=\bigcup_{\alpha\tilde{\in}\{e_1\},s\tilde{\in}\{b\}} (\alpha)\mu_s\\
=(e_1)\mu_b=\{0.8, 0.4, 0.9\}.$\\		
By similar calculations we get\\
$ f(F_A)=\{e_1'=\{<x,\{0.5\}>, <y, \{0.2, 0.4, 0.9\}>,<z,\{0.0\}> \}\\e_2'=\{<x,\{0.8, 0.4, 0.9\}>, <y,\{0.6,0.8 \}>,<z,\{0.0\}> \}\\e_3'=\{<x,\{0.2,0.6\}>, <y,\{0.4,0.8 \}>,<z,\{ 0.0\}> \}\}.$\\ 	
Again consider a hesitant fuzzy soft set $ F_B' $ in $ \overline{(V,E')} $ as\\
$F_B'=\{e_1'=\{<x,\{0.2,0.4\}>, <y, \{0.3, 0.1, 0.8\}>,<z,\{0.4\}> \}\\e_2'=\{<x,\{0.5, 0.3, 0.7\}>, <y,\{0.7 \}>,<z,\{0.7,0.8,0.2\}> \}\\e_3'=\{<x,\{0.9\}>, <y,\{0.6,0.8 \}>,<z,\{ 0.8,0.3\}> \}\}.$\\
Therefore, \\
$f^{-1}(F_B')(e_1)(a)=(q(e_1))\mu_{p(a)}=(e_2')\mu_y=\{0.7\}$\\
$f^{-1}(F_B')(e_1)(b)=(q(e_1))\mu_{p(b)}=(e_2')\mu_x=\{0.5,0.3,0.7\}$\\
$f^{-1}(F_B')(e_1)(c)=(q(e_1))\mu_{p(c)}=(e_2')\mu_y=\{0.7\}.$\\
By similar calculations, we get\\
$f^{-1}(F_B')=\{e_1=\{<a,\{0.7\}>, <b, \{0.5, 0.3, 0.7\}>,<c,\{0.7\}> \}\\e_2=\{<a,\{0.3, 0.1, 0.8\}>, <b,\{0.2,0.4 \}>,<c,\{0.3, 0.1, 0.8\}> \}\\e_3=\{<a,\{0.7\}>, <b,\{0.5,0.3,0.7 \}>,<c,\{ 0.7\}> \}\\
e_4=\{<a,\{0.6,0.8\}>, <b,\{0.9 \}>,<c,\{ 0.6,0.8\}> \}\}.$  
\end{problem}
\begin{defn}
Let $ f=(p,q)$ be a hesitant fuzzy soft mapping of a hesitant fuzzy soft class $\overline{(U,E)}$ into a  hesitant fuzzy soft class $\overline{(V,E')}.$ Then
\begin{enumerate}
\item[(i)] $ f$ is said to be a one-one (or injection) hesitant fuzzy soft mapping if for both  $ p:U\longrightarrow V $ and  $ q:E\longrightarrow E' $  are one-one.
\item[(ii)]  $ f$ is said to be a onto (or surjection) hesitant fuzzy soft mapping if for both  $ p:U\longrightarrow V $ and  $ q:E\longrightarrow E' $  are onto.
\end{enumerate}	
\end{defn}
If $ f$ is both one-one and onto then $ f$ is called a hesitant one-one onto (or bijective) correspondence of hesitant fuzzy  soft mapping. 
\begin{thm}
Let  $ f=(p,q):\overline{(U,E)}\longrightarrow\overline{(V,E')} $ and  $ g=(r,t):\overline{(U,E)}\longrightarrow\overline{(V,E')} $ are two hesitant fuzzy soft mappings. Then $ f $ and $ g $ are equal if and only if  $ p=r $ and $ q=t. $
\end{thm}
\begin{proof}
Obvious. 	
\end{proof}
\begin{thm}
Two hesitant fuzzy soft mappings $ f $ and $ g $ of a  hesitant fuzzy soft class $\overline{(U,E)}$ into a  hesitant fuzzy soft class $\overline{(V,E')}$ are equal if and only if $ f(F_A)=g(F_A) ,$ for all $ F_A\tilde{\in}\overline{(U,E)}. $
\end{thm}	
\begin{proof}
Let	 $ f=(p,q):\overline{(U,E)}\longrightarrow\overline{(V,E')} $ and  $ g=(r,t):\overline{(U,E)}\longrightarrow\overline{(V,E')} $ are two hesitant fuzzy soft mappings. Since   $ f $ and $ g $ are equal, this implies $ p=r $ and $ q=t. $\\
Let $ \beta\tilde{\in}q(E)\tilde{\subseteq}E'$ and $ y\tilde{\in}V, $\\
$ f(F_A)(\beta)(y)=\bigcup_{\alpha\tilde{\in}q^-1(\beta)\cap A,s\tilde{\in}p^-1(y)} (\alpha)\mu_s\\
=\bigcup_{\alpha\tilde{\in}t^-1(\beta)\cap A,s\tilde{\in}r^-1(y)} (\alpha)\mu_s\\
=g(F_A)(\beta)(y).$\\
Hence 	$ f(F_A)=g(F_A) .$  \\
Conversely, Let $ f(F_A)=g(F_A) ,$ for all $ F_A\tilde{\in}\overline{(U,E)}. $\\
Let $ (R,T)\tilde{\in}f, $where $ R\tilde{\subseteq}\overline{(U,E)} $ and $ T\tilde{\subseteq}\overline{(V,E')}. $\\
Therefore $ T=f(R)=g(R), $this gives $ (R,T)\tilde{\in}g .$\\
Therefore $ f\tilde{\subseteq}g. $\\
Similarly, it can be proved that $ g\tilde{\subseteq}f. $ Hence $ f=g.$
\end{proof}
\begin{defn}
If $ f=(p,q):\overline{(U,E)}\longrightarrow\overline{(V,E')} $ and  $ g=(r,t):\overline{(V,E')}\longrightarrow\overline{(W,E'')} $ are two hesitant fuzzy soft mappings, then their composite $ g\circ f $ is a hesitant fuzzy soft mapping of $ \overline{(U,E)} $ into $ \overline{(W,E'')} $ such that for every $ F_A \tilde{\in}\overline{(U,E)},(g\circ f)(F_A)=g(f(F_A)).  $we defined as\\ for $ \beta\tilde{\in}t(E')\tilde{\subseteq}E''$ and $ y\tilde{\in}W, $
$ g(f(F_A))(\beta)(y)=\bigcup_{\alpha\tilde{\in}t^-1(\beta)\cap f(A),s\tilde{\in}r^-1(y)} (\alpha)\mu_s.$
\end{defn}
\begin{problem}
From example \ref{eam35}  and consider the hesitant fuzzy soft mapping
 $ g=(r,t):\overline{(V,E')}\longrightarrow\overline{(W,E'')} ,$where $ W=\{h_1,h_2,h_3\}, E''=\{e_1'', e_3''\}  $ and \\
 $ r(x)=h_2,r(y)=h_3, r(z)=h_2 : t(e_1')=e_3'',t(e_2')=e_3'', t(e_3')=e_1''.$ \\ 
 Therefore\\
 $ g(f(F_A))(e_1'')(h_1)=\bigcup_{\alpha\tilde{\in}t^-1(e_1'')\cap f(A),s\tilde{\in}r^-1(h_1)} (\alpha)\mu_s\\
 =\bigcup_{\alpha\tilde{\in}\{e_3'\},s\tilde{\in}\phi} (\alpha)\mu_s\\
 =(e_3')\mu_\phi=\{0.0\}$\\
 $ g(f(F_A))(e_1'')(h_2)=\bigcup_{\alpha\tilde{\in}t^-1(e_1'')\cap f(A),s\tilde{\in}r^-1(h_2)} (\alpha)\mu_s\\
 =\bigcup_{\alpha\tilde{\in}\{e_3'\},s\tilde{\in}\{x,z\}} (\alpha)\mu_s\\
 =(e_3')(\mu_x\cup\mu_z)=\{0.2,0.6\}$\\
 $ g(f(F_A))(e_1'')(h_3)=\bigcup_{\alpha\tilde{\in}t^-1(e_1'')\cap f(A),s\tilde{\in}r^-1(h_3)} (\alpha)\mu_s\\
 =\bigcup_{\alpha\tilde{\in}\{e_3'\},s\tilde{\in}\{y\}} (\alpha)\mu_s\\
 =(e_3')\mu_y=\{0.4,0.8\}.$\\
By similar calculations we get\\
$ ( g\circ f)(F_A)=g(f(F_A))=\{e_1''=\{<h_1,\{0.0\}>, <h_2, \{0.2, 0.6\}>,<h_3,\{0.4,0.8\}> \}\\e_2''=\{<h_1,\{0.0\}>, <h_2, \{0.0\}>,<h_3,\{0.0\}>\}\\e_3''=\{<h_1,\{0.0\}>, <h_2, \{0.5, 0.8, 0.9\}>,<h_3,\{0.6,0.8,0.9\}> \}\}.$	
\end{problem}
\begin{thm}\label{thm311}
Let $ f=(p,q):\overline{(U,E)}\longrightarrow\overline{(V,E')} $ and  $ g=(r,t):\overline{(V,E')}\longrightarrow\overline{(W,E'')} $ are two hesitant fuzzy soft mappings. Then
\begin{enumerate}
\item[(i)] if $ f $ and $ g $ are one-one then so is $ g\circ f .$
\item[(ii)]if $ f $ and $ g $ are onto then so is $ g\circ f .$
\item[(iii)]if $ f $ and $ g $ are both bijections then so is $ g\circ f .$
\end{enumerate}	
\end{thm}	
\begin{proof}
Let us consider the hesitant fuzzy soft mappings $f=(p,q):\overline{(U,E)}\longrightarrow\overline{(V,E')}$ and
$ g=(r,t):\overline{(V,E')}\longrightarrow\overline{(W,E'')}.$
Let $ U=\{a,b,c\} $	and $ V=\{x,y,z\}, W=\{h_1,h_2,h_3\}, E=\{e_1, e_2, e_3\}, E'=\{e_1', e_2', e_3'\},  E''=\{e_1'', e_2'', e_3''\} $ and $\overline{(U,E)}, \overline{(V,E')}, \overline{(W,E'')} $ classes of hesitant fuzzy soft sets. Let $ p(a)=z, p(b)=x, p(c)=y $ ; $ q(e_1)=e_2', q(e_2)=e_3', q(e_3)=e_1'$ and $ r(x)=h_2,r(y)=h_3, r(z)=h_1 : t(e_1')=e_3'',t(e_2')=e_1'', t(e_3')=e_2''  .$   Also we  consider a hesitant fuzzy soft set $ L_A $ in $\overline{(U,E)} $ as\\
$L_A=\{e_1=\{<a,\{0.6, 0.8\}>, <b, \{0.8, 0.4, 0.9\}>,<c,\{0.3\}> \}\\e_2=\{<a,\{0.0\}>, <b,\{0.0 \}>,<c,\{0.0\}>\\e_3=\{<a,\{0.9, 0.1, 0.2\}>, <b,\{0.5 \}>,<c,\{0.2, 0.4, 0.6\}> \}\}.$\\ 
Then the hesitant fuzzy soft image of $ F_A $ under  $ f=(p,q):\overline{(U,E)}\longrightarrow\overline{(V,E')} $is obtained as \\
$ f(L_A)(e_1')(x)=\bigcup_{\alpha\tilde{\in}q^-1(e_1')\cap A,s\tilde{\in}p^-1(x)} (\alpha)\mu_s\\
=\bigcup_{\alpha\tilde{\in}\{e_3\},s\tilde{\in}\{b\}} (\alpha)\mu_s\\
=(e_3)\mu_b=\{0.5\}$\\	
$ f(L_A)(e_1')(y)=\bigcup_{\alpha\tilde{\in}q^-1(e_1')\cap A,s\tilde{\in}p^-1(y)} (\alpha)\mu_s\\
=\bigcup_{\alpha\tilde{\in}\{e_3\},s\tilde{\in}\{c\}} (\alpha)\mu_s\\
=(e_3)\mu_c=\{0.2,0.4,0.6\}$\\	
$ f(L_A)(e_1')(z)=\bigcup_{\alpha\tilde{\in}q^-1(e_1')\cap A,s\tilde{\in}p^-1(z)} (\alpha)\mu_s\\
=\bigcup_{\alpha\tilde{\in}\{e_3\},s\tilde{\in}\{a\}} (\alpha)\mu_s\\
=(e_3)\mu_a=\{0.9,0.1,0.2\}.$\\		
By similar calculations we get\\
$ f(L_A)=\{e_1'=\{<x,\{0.5\}>, <y, \{0.2, 0.4, 0.6\}>,<z,\{0.9,0.1,0.2\}> \}\\e_2'=\{<x,\{0.8, 0.4, 0.9\}>, <y,\{0.3 \}>,<z,\{0.6,0.8\}> \}\\e_3'=\{<x,\{0.0\}>, <y,\{0.0 \}>,<z,\{ 0.0\}> \}\}.$\\ 	
 Again, \\
 $ g(f(L_A))(e_1'')(h_1)=\bigcup_{\alpha\tilde{\in}t^-1(e_1'')\cap f(A),s\tilde{\in}r^-1(h_1)} (\alpha)\mu_s\\
 =\bigcup_{\alpha\tilde{\in}\{e_2'\},s\tilde{\in}\{z\}} (\alpha)\mu_s\\
 =(e_2')\mu_z=\{0.6,0.8\}$\\
 $ g(f(L_A))(e_1'')(h_2)=\bigcup_{\alpha\tilde{\in}t^-1(e_1'')\cap f(A),s\tilde{\in}r^-1(h_2)} (\alpha)\mu_s\\
 =\bigcup_{\alpha\tilde{\in}\{e_2'\},s\tilde{\in}\{x\}} (\alpha)\mu_s\\
 =(e_2')\mu_x=\{0.8,0.4,0.9\}$\\
 $ g(f(L_A))(e_1'')(h_3)=\bigcup_{\alpha\tilde{\in}t^-1(e_1'')\cap f(A),s\tilde{\in}r^-1(h_3)} (\alpha)\mu_s\\
 =\bigcup_{\alpha\tilde{\in}\{e_2'\},s\tilde{\in}\{y\}} (\alpha)\mu_s\\
 =(e_2')\mu_y=\{0.3\}.$\\
 By similar calculations we get\\
 $( g\circ f)(L_A)= g(f(L_A))=\{e_1''=\{<h_1,\{0.6,0.8\}>, <h_2, \{0.8, 0.4,0.9\}>,<h_3,\{0.3\}> \}\\e_2''=\{<h_1,\{0.0\}>, <h_2, \{0.0\}>,<h_3,\{0.0\}> \}\\ e_3''=\{<h_1,\{0.9,0.1,0.2\}>, <h_2, \{0.5\}>,<h_3,\{0.2,0.4,0.6\}> \}\}.$ \\	
Therefore\\ 
(i) From above example we see that, if $ f $ and $ g $ are one-one then so is $ g\circ f .$\\	
(ii)	From above example we see that,if $ f $ and $ g $ are onto then so is $ g\circ f .$\\
(iii) From above example we see that,	if $ f $ and $ g $ are both bijections then so is $ g\circ f .$
%\end{enumerate}		
\end{proof}
\begin{thm}
Let us consider three hesitant fuzzy soft mappings $f:\overline{(U,E)}\longrightarrow\overline{(V,E')} ,$
$ g:\overline{(V,E')}\longrightarrow\overline{(W,E'')} $ and $ h:\overline{(W,E'')}\longrightarrow\overline{(X,E''')} .$ Then $ h\circ ( g\circ f )=( h\circ g)\circ f. $
\end{thm}
\begin{proof}
Let $ F_A\tilde{\in}\overline{(U,E)} .$ Now from  definition we have, \\
$[h\circ ( g\circ f )](F_A)=h[( g\circ f )(F_A)]= h[ g(f (F_A))] $\\
Also $[(h\circ  g)\circ f ](F_A)=(h\circ g)(f (F_A))= h[ g(f (F_A))].$ 
Hence $ h\circ ( g\circ f )=( h\circ g)\circ f. $
\end{proof}
\begin{defn}
A hesitant fuzzy soft mapping  $f:\overline{(U,E)}\longrightarrow\overline{(V,E')} $ 
is said to be many one hesitant fuzzy soft mapping if two (or more than two) hesitant fuzzy soft sets in $ \overline{(U,E)} $ have the same hesitant fuzzy soft image in $\overline{(V,E')}.  $
\end{defn}
\begin{problem}
From example \ref{eam35} and consider the hesitnat fuzzy soft set $ M_A\tilde{\in}\overline{(U,E)}, $\\
	$M_A=\{e_1=\{<a,\{0.6, 0.5\}>, <b, \{0.8, 0.4, 0.9\}>,<c,\{0.6,0.8\}> \}\\e_2=\{<a,\{0.3, 0.2, 0.7\}>, <b,\{0.5 \}>,<c,\{0.4, 0.1, 0.9\}> \}\\e_4=\{<a,\{0.4\}>, <b,\{0.2,0.6 \}>,<c,\{ 0.3, 0.8\}> \}\}.$\\ 
	Then the hesitant fuzzy soft image of $ M_A $ under  $ f=(p,q):\overline{(U,E)}\longrightarrow\overline{(V,E')} $ is obtained as \\
	$ f(M_A)(e_1')(x)=\bigcup_{\alpha\tilde{\in}q^-1(e_1')\cap A,s\tilde{\in}p^-1(x)} (\alpha)\mu_s\\
	=\bigcup_{\alpha\tilde{\in}\{e_2\},s\tilde{\in}\{b\}} (\alpha)\mu_s\\
	=(e_2)\mu_b=\{0.5\}$\\	
	$ f(M_A)(e_1')(y)=\bigcup_{\alpha\tilde{\in}q^-1(e_1')\cap A,s\tilde{\in}p^-1(y)} (\alpha)\mu_s\\
	=\bigcup_{\alpha\tilde{\in}\{e_2\},s\tilde{\in}\{a,c\}} (\alpha)\mu_s\\
	=(e_2)(\mu_a\cup\mu_c)=\{0.2,0.4,0.9\}$\\
	$ f(M_A)(e_1')(z)=\bigcup_{\alpha\tilde{\in}q^-1(e_1')\cap A,s\tilde{\in}p^-1(z)} (\alpha)\mu_s\\
	=\bigcup_{\alpha\tilde{\in}\{e_2\},s\tilde{\in}\phi} (\alpha)\mu_s\\
	=(e_2)\mu_\phi=\{0.0\}$\\
	$ f(M_A)(e_2')(x)=\bigcup_{\alpha\tilde{\in}q^-1(e_2')\cap A,s\tilde{\in}p^-1(x)} (\alpha)\mu_s\\
	=\bigcup_{\alpha\tilde{\in}\{e_1, e_3\}\cap \{e_1,e_2,e_4\},s\tilde{\in}\{b\}} (\alpha)\mu_s\\
	=\bigcup_{\alpha\tilde{\in}\{e_1\},s\tilde{\in}\{b\}} (\alpha)\mu_s\\
	=(e_1)\mu_b=\{0.8, 0.4, 0.9\}.$\\		
	By similar calculations we get\\
	$ f(M_A)=\{e_1'=\{<x,\{0.5\}>, <y, \{0.2, 0.4, 0.9\}>,<z,\{0.0\}> \}\\e_2'=\{<x,\{0.8, 0.4, 0.9\}>, <y,\{0.6,0.8 \}>,<z,\{0.0\}> \}\\e_3'=\{<x,\{0.2,0.6\}>, <y,\{0.4,0.8 \}>,<z,\{ 0.0\}> \}\}.$\\ 	
Therefore  $ f(M_A)=f(F_A). $
Hence $ f $ is many one hesitant fuzzy soft mapping. 
\end{problem}
\begin{defn}
Let $i=(p,q):\overline{(U,E)}\longrightarrow\overline{(U,E)} $ be a hesitant fuzzy soft mapping, where $ p:U\longrightarrow U $ and $ q:E\longrightarrow E .$ Then $ f $ is said to be a hesitant fuzzy soft identity  mapping if both $ p $ and $ q $ are identity mappings.
\end{defn}

\begin{rem} $i=(p,q):\overline{(U,E)}\longrightarrow\overline{(U,E)} $ be a hesitant fuzzy soft identity mapping, then  $i(F_A)=F_A  ,$ where $ F_A\tilde{\in}\overline{(U,E)}. $ 
\end{rem}
\begin{thm}
Let $f=(p,q):\overline{(U,E)}\longrightarrow\overline{(V,E')} $ be a hesitant fuzzy soft mapping and let $i=(l,m):\overline{(U,E)}\longrightarrow\overline{(U,E)} $ and  $j=(g,h):\overline{(V,E')}\longrightarrow\overline{(V,E')} $ are  hesitant fuzzy soft identity mappings then $ f\circ i =f $ and 	$ j\circ f =f .$
\end{thm}
\begin{proof} Consider the following example. 	
We consider $ F_A $ from example \ref{eam35}  and consider the hesitant fuzzy soft mappings $i=(p,q):\overline{(U,E)}\longrightarrow\overline{(U,E)}, $ where  $ p:U\longrightarrow U $ and $ q:E\longrightarrow E ,$ such that 
$ p(a)=a, p(b)=b, p(c)=c $ ; $ q(e_1)=e_1, q(e_2)=e_2, q(e_3)=e_3,q(e_4)=e_4. $ Therefore,\\
$ i(F_A)(e_1)(a)=\bigcup_{\alpha\tilde{\in}q^-1(e_1)\cap A,s\tilde{\in}p^-1(a)} (\alpha)\mu_s\\
=\bigcup_{\alpha\tilde{\in}\{e_1\},s\tilde{\in}\{a\}} (\alpha)\mu_s\\
=(e_1)\mu_a=\{0.6,0.8\}$\\	
$ i(F_A)(e_1)(b)=\bigcup_{\alpha\tilde{\in}q^-1(e_1)\cap A,s\tilde{\in}p^-1(b)} (\alpha)\mu_s\\
=\bigcup_{\alpha\tilde{\in}\{e_1\},s\tilde{\in}\{b\}} (\alpha)\mu_s\\
=(e_1)(\mu_b)=\{0.8,0.4,0.9\}$\\	
$ i(F_A)(e_1)(c)=\bigcup_{\alpha\tilde{\in}q^-1(e_1)\cap A,s\tilde{\in}p^-1(c)} (\alpha)\mu_s\\
=\bigcup_{\alpha\tilde{\in}\{e_1\},s\tilde{\in}\{c\}} (\alpha)\mu_s\\
=(e_1)\mu_c=\{0.3\}.$\\
By similar calculations we get\\
$ i(F_A)=\{e_1=\{<a,\{0.6, 0.8\}>, <b, \{0.8, 0.4, 0.9\}>,<c,\{0.3\}> \}\\e_2=\{<a,\{0.9, 0.1, 0.2\}>, <b,\{0.5 \}>,<c,\{0.2, 0.4, 0.6\}> \}\\e_3=\{<a,\{0.0\}>, <b,\{0.0 \}>,<c,\{0.0\}>\\e_4=\{<a,\{0.3\}>, <b,\{0.2,0.6 \}>,<c,\{ 0.4, 0.8\}> \}\}.$\\ 	
Hence $i(F_A)=F_A\Rightarrow f(i(F_A))=f(F_A)\Rightarrow (f\circ i)(F_A) =f(F_A) \Rightarrow f\circ i =f.$\\
Similarly we get $ f(F_A)\tilde{\in}\overline{(V,E')} $ and 
 $ j(f(F_A))=f(F_A)\Rightarrow(j\circ f)(F_A)=f(F_A). $ \\
Hence
$ j\circ f =f. $
\end{proof}
\begin{defn}
 A one-one onto hesitant fuzzy soft mapping $f=(p,q):\overline{(U,E)}\longrightarrow\overline{(V,E')} $ is called hesitant fuzzy soft invertable mapping. Its hesitant fuzzy soft inverse mapping is denoted by $f^{-1}=(p^{-1},q^{-1}):\overline{(V,E')}\longrightarrow\overline{(U,E)}. $
\end{defn}
\begin{rem}
In a hesitant fuzzy soft invertable mapping $f:\overline{(U,E)}\longrightarrow\overline{(V,E')} , $ for $ F_A \tilde{\in}\overline{(U,E)},G_B \tilde{\in}\overline{(V,E')}, $ 
$ f^{-1}(G_B)=F_A,$ whenever  $f(F_A)=G_B. $
\end{rem}

\begin{problem}
From proof of the Theorem \ref{thm311}. We consider $ f(L_A)=G_B. $ Therefore, \\
$ f^{-1}(G_B)(e_1)(a)=(q(e_1))\mu_{p(a)}=(e_2')\mu_z=\{0.6,0.8\} $\\$ f^{-1}(G_B)(e_1)(b)=(q(e_1))\mu_{p(b)}=(e_2')\mu_x=\{0.8,0.4,0.9\} $\\
$f^{-1}(G_B)(e_1)(c)=(q(e_1))\mu_{p(c)}=(e_2')\mu_y=\{0.3\} .$\\
By similar calculations we get\\
$f^{-1}(G_B)=\{e_1=\{<a,\{0.6, 0.8\}>, <b, \{0.8, 0.4, 0.9\}>,<c,\{0.3\}> \}\\e_2=\{<a,\{0.0\}>, <b,\{0.0 \}>,<c,\{0.0\}>\\e_3=\{<a,\{0.9, 0.1, 0.2\}>, <b,\{0.5 \}>,<c,\{0.2, 0.4, 0.6\}> \}\}.$\\ 
Hence $ f^{-1}(G_B)=L_A. $
\end{problem}
\begin{thm}
Let  $f:\overline{(U,E)}\longrightarrow\overline{(V,E')} $ be a hesitant fuzzy soft invertable mapping. Therefore its hesitant fuzzy soft inverse mapping is unique.  
\end{thm}
\begin{proof}
Let $  f^{-1}$ and $  g^{-1}$ are two hesitant fuzzy soft inverse mappings of $ f. $ Therefore,\\
$ f^{-1}(G_B)=F_A,$ whenever  $f(F_A)=G_B, F_A\tilde{\in}\overline{(U,E)},  G_B\tilde{\in}\overline{(V,E')},$ and\\
$ g^{-1}(G_B)=H_A,$ whenever  $g(H_A)=G_B, H_A\tilde{\in}\overline{(U,E)},  G_B\tilde{\in}\overline{(V,E')} .$ 
Thus $  f(F_A)=g(H_A). $\\
Since $ f $ is one-one, therefore $  F_A=H_A. $
Hence $ f^{-1}(G_B)= g^{-1}(G_B)  $ i.e $  f^{-1}=g^{-1}.$ 
\end{proof}
\begin{thm}
Let  $f:\overline{(U,E)}\longrightarrow\overline{(V,E')} $ and $g:\overline{(V,E')}\longrightarrow\overline{(W,E'')} $ are two one-one onto hesitant fuzzy soft mappings. If  $f^{-1}:\overline{(V,E')}\longrightarrow\overline{(U,E)} $ and $g^{-1}:\overline{(W,E'')}\longrightarrow\overline{(V,E')} $ are hesitant fuzzy soft inverse mappings of $ f $ and $ g, $
respectively, then the  inverse of the mapping  $g\circ f:\overline{(U,E)}\longrightarrow\overline{(W,E'')} $ is the hesitant fuzzy soft mapping $f^{-1}\circ g^{-1}:\overline{(W,E'')}\longrightarrow\overline{(U,E)}. $ 
\end{thm}
\begin{proof}
Obvious.
\end{proof}
\begin{thm}
A hesitant fuzzy soft mapping $f:\overline{(U,E)}\longrightarrow\overline{(V,E')} $ is invertable  if and only if there exists a hesitant fuzzy soft inverse mapping $f^{-1}:\overline{(V,E')}\longrightarrow\overline{(U,E)}$ such that  
$ f^{-1}\circ f =i_{\overline{(U,E)}} $ and $ f\circ f^{-1} =i_{\overline{(V,E')}}, $ where $ i_{\overline{(U,E)}} $ and $ i_{\overline{(V,E')}} $ is  hesitant fuzzy soft identity mapping on $ \overline{(U,E)} $  and $\overline{(V,E')}, $ respectively. 
\end{thm}
\begin{proof}
 Let  $f:\overline{(U,E)}\longrightarrow\overline{(V,E')} $ be a hesitant fuzzy soft  invertable mapping. Therfore by definition we have\\
  $ f^{-1}(G_B)=F_A, whenever  f(F_A)=G_B, F_A\tilde{\in}\overline{(U,E)},  G_B\tilde{\in}\overline{(V,E')}.$\\
Since $ (f^{-1}\circ f)(F_A)=f^{-1}(f(F_A))=f^{-1}(G_B)=F_A.$
Therefore $ f^{-1}\circ f =i_{\overline{(U,E)}} .$\\
Similarly, we prove that $ f\circ f^{-1} =i_{\overline{(V,E')}}.$ 
\end{proof}
\begin{thm}
If  $f:\overline{(U,E)}\longrightarrow\overline{(V,E')} $  and  $g:\overline{(V,E')}\longrightarrow\overline{(W,E'')} $ are  two one-one onto hesitant fuzzy soft mapping then $  (g\circ f)^{-1}= f^{-1}\circ g^{-1} .$
\end{thm}
\begin{proof}
Since $ f $ and $ g $ are  one-one onto hesitant fuzzy soft mapping, then there exists  $f^{-1}:\overline{(V,E')}\longrightarrow\overline{(U,E)}$ and $g^{-1}:\overline{(W,E'')} \longrightarrow\overline{(V,E')}$ such that \\
 $ f^{-1}(G_B)=F_A, whenever  f(F_A)=G_B, F_A\tilde{\in}\overline{(U,E)},  G_B\tilde{\in}\overline{(V,E')} ,$ and\\
 $ g^{-1}(H_C)=G_B, whenever  g(G_B)=H_C, H_C\tilde{\in}\overline{(W,E'')},  G_B\tilde{\in}\overline{(V,E')}.$\\
Therefore, $  (g\circ f)(F_A) =g[f(F_A)]=g(G_B)=H_C .$\\ 
As $ g\circ f $ is one-one onto, therefore $ (g\circ f)^{-1}  $ exists such that \\
$  (g\circ f)(F_A)=H_C\Rightarrow (g\circ f)^{-1}(H_C)=F_A. $\\ 
Also  $  (f^{-1}\circ g^{-1})(H_C) =f^{-1}[g^{-1}(H_C)]=f^{-1}(G_B)=F_A.$\\
Hence $(g\circ f)^{-1}(H_C)= (f^{-1}\circ g^{-1})(H_C) \Rightarrow (g\circ f)^{-1}= f^{-1}\circ g^{-1} . $
\end{proof}

\end{document}